\documentclass[10pt]{article}
\usepackage[utf8]{inputenc}
\usepackage{amssymb}
\usepackage{amsmath}
\usepackage{amsthm}
\usepackage{tikz}
\usepackage{circuitikz}
\usepackage{hyperref}
\usetikzlibrary{arrows,calc,decorations.pathreplacing,fadings,intersections}
\usetikzlibrary{external}

\newtheorem{theorem}{Theorem}
\newtheorem{lemma}[theorem]{Lemma}
\newtheorem{corollary}[theorem]{Corollary}
\newtheorem{proposition}[theorem]{Proposition}
\newtheorem{definition}{Definition}

\newtheorem{remark}{Remark}

\DeclareMathOperator{\proj}{proj}

\newcommand{\dd}{{\rm d}}

\newcommand{\diam}{\mathrm{diam}}

\newcommand{\vol}{\mathrm{vol}}

\title{
Illumination bodies  for ball-convex bodies}
\author{Brandon E. Oliva\footnote{Department of Mathematics, Case Western Reserve University, 2145 Adalbert Road, Cleveland, OH 44106, USA, {\tt beo20@case.edu}} \hskip 7mm
Elisabeth M. Werner  \footnote{Department of Mathematics, Case Western Reserve University, 2145 Adalbert Road, Cleveland, OH 44106, USA,
{\tt elisabeth.werner@case.edu}} \hskip 7mm Diliya Yalikun 
\footnote{Department of Mathematics, Case Western Reserve University, 2145 Adalbert Road, Cleveland, OH 44106, USA, {\tt dxy259@case.edu}} 
}
\date{}

\begin{document}

\maketitle

\begin{abstract}
We introduce illumination bodies and weighted illumination bodies in the class of $n$-dimensional $R$-ball convex bodies. 
These bodies may be regarded as dual counterparts of the recently introduced $R$-ball floating bodies. 
We prove that the illumination bodies are convex. We also show that a left derivative of  volume gives rise to a 
surface area measure for ball-convex bodies, the relative surface area measure. In dimension $2$, we 
establish an isoperimetric inequality for this relative surface area.
\end{abstract}

    

\par

\section{Introduction}

A convex body in $\mathbb{R}^n$ is a compact convex subset of $\mathbb{R}^n$ with nonempty interior. Let $C\subset \mathbb{R}^n$ be a convex body. A set $K\subset \mathbb{R}^n$ is called $C$-ball convex if it is the intersection of all translates of $C$ that contain it, that is, 
\[ K=\cap_{K\subset C+x}(C+x).  \] 
This class was introduced and studied in \cite{BlaschkeKK, LangiNazodiTalata}. In the important special case $C=R\, B_2^n$, where $B_2^n$ is the Euclidean unit ball centered at the origin, we call $K$ $R$-ball convex and denote the corresponding class by $\mathcal {K}_R$.
Ball-convex bodies have been investigated from several points of view; see, for instance, \cite{ArtsteinFlorentin, BezdekConnellyCsikos, BezdekNazodi, BezdekLangiNazodi,  KabluchkoMarynychMolchanov}. They are connected with optimal transport, the Kneser–Poulsen conjecture, Blaschke’s rolling theorem, 
 with Meissner polyhedra,  bodies of constant width, 
and isoperimetric problems; see  \cite{ArtsteinSadovskyWyczensany, Bezdek,  Drach, DrachTatarko, Hynd1, Hynd2, MartiniMontejanoOliveros} and the references therein. 
Isometries of classes of ball-convex bodies were classified in  \cite{ArtsteinChorFlorentin1, ArtsteinChorFlorentin2}. 
 For a nice recent survey on ball convexity we refer to \cite{BezdekLangiNazodi26}.
\vskip 2mm
Important closely related notions in convex geometry are the (convex) floating
body, the illumination body and the affine surface area of a convex body.
The floating body of a convex body is obtained by cutting off hyperplane caps 
of volume less than or equal to a fixed positive constant $\delta$.
The illumination body,  introduced in \cite{Werner94}, is  the set of those points whose convex hull with a given convex body  have fixed volume.
The first-order volume asymptotics of both the floating body and the illumination body lead to the affine surface area.
For the floating body this was established for all convex bodies in all dimensions 
in \cite{SW:1990} and for the illumination body this was shown in  \cite{Werner94}.
The affine surface area is a fundamental affine invariant associated with convex bodies in Euclidean space. It was first introduced by Blaschke~\cite{Blaschke:1923} for convex bodies in low dimensions under additional regularity assumptions, and later extended to arbitrary dimensions without such assumptions; see, for example, the survey~\cite{SW:2023}. 
\begin{equation}\label{asa}
as(K) = \int_{\partial K} \kappa(K, x)^\frac{1}{n+1} d\mu_K(x),
\end{equation}
where $\kappa(K,x) $ is the Gauss curvature of $x \in \partial K$, the boundary of $K$ and $\mu_K$ is the usual surface area measure on $\partial K$.
 The  importance of affine surface area is reflected in the affine isoperimetric inequality, in extremal and approximation problems, and in stochastic geometry; see  \cite{Boe1, PSchW:2022, Reitz1, Reitzner, SW4, TW1,TW2, TW4}. Many extensions have since been developed, including $L_p$-, Orlicz-, and stochastic affine surface areas \cite{Lutwak:1996, SW:2004, TW:2023, XZ:2022, Ye:2015, Ye:2016, Zawalski:2025, Zhao:2016}. 
Analogues of floating bodies and affine surface area have also been established in spherical and hyperbolic spaces \cite{BW:2016, BW:2018, BesauWerner1, BesauWerner}.  Very recently, illumination bodies have been 
introduced and studied in spaces of constant curvature  \cite{ABW} and on Riemannian manifolds \cite{ASW}.
\vskip 2mm
It is natural to ask for an analogous theory in the setting of ball convexity. A first step in this direction was the recent introduction of floating bodies for $C$-ball convex bodies  \cite{SWY, WernerYalikun}. 
It was shown in \cite{SWY} that in the Euclidean case $C=R\, B_2^n$, their volume asymptotics give rise to the {\em relative surface area}
\begin{equation}\label{Rasa}
as^R(K)= \int_{\partial K} \prod _{i=1}^{n-1} \left( \kappa_i (K, x) -\frac{1}{R} \right)^\frac{1}{n+1} d \mu_K(x),
 \end{equation}
where the $\kappa_i(K, x)$ are the principal curvatures of $K$ at $x$, see below.
 The factor 
\[ \prod_{i=1}^{n-1} \left(\kappa_i(K,x)-\frac{1}{R}\right)
\] is the natural {\em relative Gauss curvature} in the $R$-ball convex setting. As $R\to \infty$, $as^R$ recovers the classical affine surface area (\ref{asa}). Moreover, it shares several structural properties with affine surface area, including rigid motion invariance, valuation property, and upper semicontinuity \cite{SWY}.  Its appearance in approximation problems  \cite{FodorKeveiVigh, FodorPapvari, FodorGrunfelder} further indicates that it is the natural surface-area functional for $R$-ball convex bodies.
\vskip 2mm
The aim of the present paper is to develop the exterior counterpart. We introduce $R$-ball illumination bodies by replacing the usual convex hull in the classical illumination construction with the $R$-ball convex hull. 
This produces an exterior family of bodies adapted to the geometry of $R$-ball convexity.
Our first result shows that $R$-ball illumination bodies are convex. We also give an example showing that they need not be $R$-ball convex. 
\par
Our main asymptotic theorem shows that the left derivative of their volume gives  precisely the relative surface area measure (\ref{Rasa}). Hence the classical duality between floating bodies and illumination bodies persists in the ball-convex setting: the inward and outward constructions lead to the same curvature measure, with the 
Gauss curvature 
\[
\prod_{i=1}^{n-1} \kappa_i (K,x)
\] replaced by the  relative Gauss curvature
\[ 
\prod_{i=1}^{n-1} \left(\kappa_i(K,x)-\frac{1}{R}\right).
\]
We also introduce weighted $R$-ball illumination bodies. Their volume derivatives yield the relative $L_p$ surface areas introduced in \cite{WernerYalikun}, which are the $R$-ball analogues of Lutwak’s $L_p$ 
affine surface areas \cite{Lutwak:1996, SW:2004}. Finally, in dimension two, we prove an isoperimetric inequality for the relative surface area.
\vskip 4mm
\noindent
{\bf Further notation.} 
The  closed Euclidean ball centered at $a$ with radius $r$ is $B^n_2(a,r)$. We write in short $B^n_2=B^n_2(0,1)$ and $S^{n-1}= \partial B^n_2$. The Euclidean norm on $\mathbb{R}^n$
is $\| \cdot\|$.
We denote by  $H\left(x, u\right)$ the hyperplane through $x$ orthogonal to the vector $u$. The line segment joining $x$ and $y$ is $[x,y]$.
For a convex body $K$ in $\mathbb{R}^n$, $\vol_n(K)$ denotes its volume. 
By $N(x)$ or $N_K(x)$ we denote the unit outer normal to $K$ in the boundary point $x$.
For more information and details on convex bodies we refer to e.g., the books  \cite{Gardner,  SchneiderBook}.
\par
\noindent
Finally,  $c$, $c_1$, $d$,  $d_1$, etc.,  are absolute constants that may change from line to line.

\vskip 4mm
\noindent
{\bf Acknowledgements.}
Elisabeth Werner was supported by NSF grant DMS-2506790 and by Simons Fellowship SFI-MPS-SFM-00020853. She also 
wants to thank the University of M\"unster (Germany) for their hospitality. It was during her stay there as a M\"unster Research Fellow that part of the work was completed.
\par
\noindent
Diliya Yalikun was supported by NSF grant DMS-2506790.

 \medskip
 \noindent
    \textbf{Data availability statement}.
    Data sharing not applicable to this article, as no datasets were generated or analyzed during the current study. The corresponding author declares, on behalf of all authors, that there are no conflicts of interest.

\section{Results}

Let $K$ be a convex body in $\mathbb{R}^n$. For a point $x$ on the boundary $\partial K$ of $K$,  we denote by $N_K(x)$ the outer unit normal vector of $K$ at $x$. 
By a theorem of Rademacher,  e.g., \cite{Rademacher}, $N_K(x)$ is unique a.e. on $\partial K$. The map $N: \partial K \to S^{n - 1}$, $x \to N(x)=N_K(x)$,  is called the spherical image map or Gauss map.  
Its  differential resp. generalized differential is called the Weingarten map. By a theorem of Alexandrov \cite{Alexandroff} and Busemann and Feller \cite{Buse-Feller}
these generalized derivatives exist a.e. on $\partial K$.
The eigenvalues of the Weingarten map are the principal curvatures $\kappa_i(K, x)$ of $K$ at $x$.
The (generalized) Gauss curvature is $\kappa(K, x) = \prod_{i=1}^{n-1} \kappa_i(K,x)$.
For more information and the details we refer to e.g., \cite{Gardner,  SchneiderBook}.
\vskip 4mm
Let $\xi \in S^{n-1}$ and $\alpha \in \mathbb{R}$. Then $H=\{x \in \mathbb{R}^n: \langle x, \xi \rangle=\alpha\}$ is the hyperplane orthogonal to $\xi$ at distance $\alpha$
to the origin and $H^+=\{x \in \mathbb{R}^n: \langle x, \xi \rangle \geq \alpha\}$ and $H^-=\{x \in \mathbb{R}^n: \langle x, \xi \rangle\leq\alpha\}$ are the two closed halfspaces
determined by $H$.
We recall that for a convex body $K$ in $\mathbb{R}^n$, for $\delta \geq 0$, the  illumination  body \cite{Werner94}  is defined as 
\begin{equation}\label{illubody}
K^{\delta}=\{x\in\mathbb{R}^n: \vol_n(\mathrm{conv}[K,x] \setminus K)\leq \delta \}\quad.
\end{equation}
Note that the illumination body is always convex, see e.g., \cite{MordhorstW}. 
\vskip 4mm
\noindent
We now introduce  an analog to this illumination  body for $C$-ball convex bodies, the $C$-ball illumination body.  
We need the $C$-convex hull  of a $C$-ball convex body $K$ and a point $x$ which is defined as follows, see e.g., \cite{LangiNazodiTalata},
\begin{equation}\label{spindlehull}
[K,x] _C= \bigcap_{K\cup \{x\} \subseteq (C+y)}C+y.
\end{equation}
When $C=RB^n_2$, we write in short $[K,x] _R= [K,x] _{R B^n_2}$. 
\vskip 2mm
\begin{definition}\label{def:R-illu} 
Let $\delta \geq 0$. Let $K$ be a $C$-ball convex body.
 We define the $C$-ball illumination  body $I_C^\delta(K)$ by 
\begin{equation}\label{L-float}
I_C^\delta(K) = \{ x \in \mathbb{R}^n: \vol_n([K,x]_C \setminus K)\leq \delta\}.
\end{equation}
In particular,  when $C=R\, B^n_2$, we  
 write in short $I_{RB^n_2}^\delta(K)=I_R^\delta(K)$ and call it the $R$-ball illumination  body of $K$.
\end{definition}
\vskip 3mm
\noindent
We will also consider weighted versions which we introduce next.
\vskip 2mm
\begin{definition}\label{def:R-illu-weighted} 
Let $K$ be a $C$-ball convex body. Let  $U \supset K$ be an open set. Let $\varphi : U \to (0,\infty)$ be a continuous and integrable function. Let $A$ be a measurable set in $\mathbb{R}^n$. We define a measure $\vol_n^\varphi$ on $\mathbb{R}^n$ by
\[
    \vol_n^\varphi(A)
    = \int_{A \cap U} \varphi \, \dd \lambda_n,
\]
where $\lambda_n$ denotes Lebesgue measure.
For $\delta \geq 0$, the weighted R-illumination body $I_{C,\varphi}^\delta(K)$ is defined by
\[
    I_{C, \varphi}^\delta(K)
    = \left\{ x \in \mathbb{R}^n : \vol_n^\varphi([K, x]_C\setminus K) \leq \delta \right\}.
\]
In particular,  when $C=R\, B^n_2$, we  
 write in short $I_{RB^n_2, \varphi}^\delta(K)=I_{R, \varphi}^\delta(K)$ and call it the  weighted $R$-ball illumination  body of $K$.
\end{definition}
\vskip 2mm
\noindent
Obviously, when $\varphi \equiv 1$, $I_{C,\varphi}^\delta(K)= I_C^\delta(K)$.  Moreover, $K \subseteq I_{C,\varphi}^\delta(K)$. 
In this paper we will mostly treat the case $C=R\, B^n_2$.  Despite a slight abuse of notation, there should be 
no issue if we   
 write in short  $I_{R\, B^n_2,\varphi}^\delta(K) =  I_{R,\varphi}^\delta(K)$ and $I_{R\, B^n_2}^\delta(K)=I_R^\delta(K)$. 
\vskip 3mm
\noindent
 In the definition of the $R$-ball illumination  body  the $R$-ball convex hull  replaces the usual convex hull  of the definition of the classical illumination body (\ref{illubody}). 
 When  $R \to \infty$, then $ [K,x]_R\to [K,x]$ and we recover the usual illumination body \cite{Werner94} $K^\delta$,  resp. in the weighted case,  the weighted versions $
 \mathcal{I}^\varphi_\delta(K)$ of \cite{ABW}.
Observe that 
\begin{equation}\label{inclusion}
I_R^\delta(K) \subset K^\delta  \hskip 3mm \text{and} \hskip 3mm I_{R, \varphi}^\delta(K)  \subset \mathcal{I}^\varphi_\delta(K).
\end{equation}
\par
\noindent
Note also that we  have for
all  affine maps $T$ on $\mathbb{R}^n$ with determinant $\mbox{det}(T)\neq 0$ 
that $K$ is $R$-ball convex if and only if $T(K)$ is $T(R\, B^n_2)$-ball convex
and that for all $\delta \geq 0$ 
\begin{equation*}\label{Affine:map:floating:body}
I^\delta_{T (RB^n_2)}(T(K))=T\left(I_R^{\frac{\delta} {| det (T)|}} (K)\right) \hskip 3mm \text{and} \hskip 3mm  I_{T (R B^n_2),  \varphi \circ T^{-1}}^\delta(T K) = T(I_{R,  \varphi}^\frac{\delta}{| det (T)|}( K)).
\end{equation*}
\vskip 3mm
\noindent
{\bf Remarks and Examples.}
\vskip 2mm
\noindent
(i) We show in the next proposition that the $R$-ball illumination body is always convex. The $R$-ball illumination body of an  $R$-ball convex body is not $R$-ball convex in general.
We give  examples. 
\par
\noindent
a) The easiest example is  $R B^n_2$. It is $R$-ball convex and its $R$-ball illumination body is $(R+ \alpha_n \delta^\frac{1}{n+1}) B^n_2$ which is not $R$-ball convex 
but $\tilde{R}$-ball convex for all $\tilde{R} \geq (R+ \alpha_n \delta^\frac{1}{n+1})$. 
\par
\noindent
b) Let $R=1$ and let $0 < \varepsilon <1$. Let $L_\varepsilon$ be the lens in $\mathbb{R}^2$, 
\begin{equation}\label{lens}
L_\varepsilon = B^2_2((-\varepsilon,0), 1) \cap B^2_2((\varepsilon,0), 1).
\end{equation}
Then $L_\varepsilon$ is $1$-ball convex.
Let 
$$ 
x_+ =\left(1+\frac{\varepsilon}{2},0\right), \hskip 4mm x_- =\left(-1-\frac{\varepsilon}{2},0\right)
$$
and let $\delta_\varepsilon$ be such that 
\begin{equation*} 
\delta_\varepsilon= \vol_2\left( [x_{\pm}, L _\varepsilon]_R\setminus L_\varepsilon \right).
\end{equation*}
Then $x_+ \in \partial B^2_2((\frac{\varepsilon}{2},0),1)$ and $L_\varepsilon \subset B^2_2((\frac{\varepsilon}{2},0),1)$ and we have by construction that 
\begin{equation*}
x_{\pm} \in \partial I_R^{\delta_{\varepsilon}}(L_\varepsilon).
\end{equation*}
This means that 
\begin{equation*}
\diam  \left( I_R^{\delta_{\varepsilon}}(L_\varepsilon)\right)  \geq \|x_+ - x_-\| =2+\varepsilon >2.
\end{equation*}
But for every  $1$-ball convex body $K$ we have that $\diam (K) \leq 2$. Therefore $I_R^{\delta_{\varepsilon}}(L_\varepsilon)$ is not $1$-ball convex.
Note however that $I_R^{\delta_{\varepsilon}}(L_\varepsilon)$ is e.g., $(1+2 \varepsilon)$-ball convex. We explain this in Proposition \ref{R-convex}.
\vskip 2mm
\noindent
(ii) We cannot expect that the {\em weighted} $R$-illumination body is convex in general. Examples are given in \cite{ABW} for the weighted illumination body $\mathcal{I}^\varphi_\delta(K)$.
Such an example, appropriately modified, applies in our situation as well. To see that, let  $U=(-2,2)^2$, $a=R-1$,  and 
  $$B_1=B^2_2((0,-a), R), \hskip 2mm B_2=B^2_2 ((0,a), R),  \hskip 2mm B_3=B^2_2((-a,0), R), \hskip 2mm B_4=B^2_2((a,0), R).$$ 
  Let $S= \cap _{i=1}^4 B_i$ be $R$-ball polytope corresponding to the square in $\mathbb{R}^2$.       
We consider the weight $\varphi \equiv 1$ on $U$. For $e_1=(1,0)$ and $p_1 = (2R-1)^{\frac{1}{2}}$,  we have that $ \vol_2^{\varphi}([p_1e_1,S]_R\setminus S) \sim 2-\frac{c_1}{R}$, where $c_1>0$ is an absolute constant.
Therefore $\rho(I_{R,\varphi}^\delta(S),p_1 e_1)$ can be made arbitrarily large for $\delta\geq 2-\frac{c_1}{R}$. Here $\rho(K,u) = \max \{\lambda \geq 0: \lambda u \in K\}$ is the radial function of $K$.
For $u=(\frac{1}{\sqrt{2}},\frac{1}{\sqrt{2}})$ we have for a finite $p_2 < 1$,  $\vol_2^\varphi([p_2 \, u ,S]_R\setminus S) \sim  4(1-\frac{c_2}{R})$, where $c_2>0$ is an absolute constant and 
therefore $\rho(I_{R,\varphi}^\delta(S),p_2 \, u) <\infty$ if $\delta < 4(1-\frac{c_2}{R})$.
So for $\delta\in [2-\frac{c_1}{R},4(1- \frac{c_2}{R}))$ 
the weighted illumination body $I_{R,\varphi}^\delta(K)$ is  not convex.
\par
\noindent
However, for an $R$-ball convex body with $0$ in its interior, the weighted $R$-illumination body is always star shaped by definition. 

\vskip 3mm
\begin{proposition}\label{convex}
Let $K$ be an $R$-ball convex body and let $I_R^\delta(K)$ be its $R$-illumination body.
Then $I_R^\delta(K)$ is convex.
\end{proposition}
\vskip 3mm
\noindent
The following theorems are the main results  of the  paper. We denote by 
$\mathcal{K}_R^+$ the set of all $R$-ball convex bodies in $\mathbb{R}^n$
such that for all $x \in \partial K$, $\kappa_i(K,x) >\frac{1}{R}$, for all $1 \leq i \leq n-1$. 
\vskip 2mm
\begin{theorem}\label{theorem:limit} 
Let $K \in \mathcal{K}_R^+$ and let $I_R^\delta(K)$ be its $R$-ball illumination body.  Then
$$
\lim_{\delta \to 0} \frac{\vol_n\left(I_R^\delta(K)\right)-\vol_n(K) } {\delta^\frac{2}{n+1}} =  d_n \int_{\partial K} \prod _{i=1}^{n-1} \left( \kappa_i (K, x) -\frac{1}{R} \right)^\frac{1}{n+1} d \mu_K(x), 
$$
where $d_n=\frac{1}{2} \left( \frac{n(n+1)}{\vol_{n-1}(B^{n-1}_{2})}\right)^\frac{2}{n+1}$.
\end{theorem}
\vskip 3mm
\noindent
Note that if $R \to \infty$, we recover on the right hand side of the theorem the classical affine surface area (\ref{asa}).
This led  naturally to calling  the expression on the right hand side  of the above theorem {\em relative surface area} in \cite{SWY},
\begin{equation*}\label{Rasa}
as^R(K)=\int_{\partial K} \prod _{i=1}^{n-1} \left( \kappa_i (K, x) -\frac{1}{R} \right)^\frac{1}{n+1} d \mu_K(x).
\end{equation*}
\vskip 3mm
\noindent
A derivative of volume involving the weighted $R$-ball illumination body is given in the next theorem.
\vskip 2mm
\begin{theorem}\label{theorem:weighted limit} 
Let $K \in \mathcal{K}_R^+$ and let  $U \supset K$ be an open set. Let $\varphi : U \to (0,\infty)$ be a continuous and integrable function.
Let  $I_{R, \varphi}^\delta(K)$ be the weighted $R$-ball illumination  body of $K$. Then
$$
\lim_{\delta \to 0} \frac{\vol_n\left( I_{R, \varphi}^\delta(K)\right)-\vol_n(K) } {\delta^\frac{2}{n+1}} =  
d_n \int_{\partial K} \prod _{i=1}^{n-1} \left( \kappa_i (K, x) -\frac{1}{R} \right)^\frac{1}{n+1} \varphi(x)^{-\frac{2}{n+1}}\, d \mu_K(x), 
$$
where $d_n$ is as above.
\end{theorem}
\vskip 3mm
\noindent
An extension of the classical affine surface area (\ref{asa})  to the setting of Lutwak's  $L_p$ Brunn-Minkowski theory was given by Lutwak. He introduced 
the $L_p$-affine surface areas \cite{Lutwak:1996}  (see also \cite{SW:2004}).  For a convex body $K$ in $\mathbb{R}^n$ and $p \in \mathbb{R}$, $p \neq -n$, they are defined as
\begin{equation*}
    as_{p}(K) = \int_{\partial K} \frac{\kappa(K, x)^\frac{p}{n+p} } {\langle x, N_K(x) \rangle ^ \frac{n(p-1)}{n+p}}  \, d \mu_K(x).
\end{equation*}
The case $p=1$ is the classical affine surface area (\ref{asa}). 
The $L_p$-affine surface areas are semicontinuous,  linear invariant valuations \cite{Ludwig, LR:2010, Lutwak:1996}, they satisfy affine isoperimetric inequalities \cite{Lutwak:1996,LYZ:2000, WY:2008} and  they have a Steiner formula associated with them \cite{TW:2019, TW:2023}.
\vskip 3mm
\noindent
An extension of $L_p$-affine surface area to  {\em relative $L_p$ surface area} $\Omega_p^R(K)$ in the setting of $R$-ball convexity was given in \cite{WernerYalikun} as follows.
\par
\noindent
Let $K$ be an $R$-ball convex body. Let $ - \infty \leq p \leq \infty$, $p\neq -n$. Then 
\begin{equation} \label {p-Rasa}
\Omega_p^R(K) = \int_{\partial K}\Big (\frac{\kappa(K, x)^\frac{1}{n+1}}{\langle x, N(x)\rangle}\Big )^\frac{n(p-1)}{n+p}  \prod _{i=1}^{n-1} \left( \kappa_i (K, x) -\frac{1}{R} \right)^\frac{1}{n+1} d \mu_K(x).
\end{equation}
Note that for $p=1$ we recover the relative surface area (\ref{Rasa}) and that for $R\to \infty$, we recover the ``classical" $L_p$ affine surface areas.
\vskip 2mm
\noindent
The relative $L_p$-surface areas can be obtained as a  corollary to Theorem \ref {theorem:weighted limit}.
Indeed, let $K$ be a convex body in $\mathbb{R}^n$ that is $C^2_+$.
For $p \in \mathbb{R}$, $p \neq -n$, we specify   $\varphi$
to $\varphi_p$  such that for $x \in \partial K$,
\begin{equation}\label{varphip}
 \varphi_p(x)^{-\frac{2}{n+1}}
    = \Big (\frac{\kappa(K, x)^\frac{1}{n+1}}{\langle x, N_K(x)\rangle}\Big )^\frac{n(p-1)}{n+p} 
\end{equation}
and extend $\varphi_p$  continuously to $U$.
The next corollary shows that  the  relative $L_p$  surface area is  a right-derivative at $\delta=0$.
\par
\begin{corollary}
Let  $K$ be a convex body in $\mathbb{R}^n$ that is $C^2_+$ and let $\varphi_p$ be as in (\ref{varphip}). Then
\[
\lim_{\delta\to 0} \frac{\vol_n (I_{R, \varphi_p}^{\delta}(K))- \vol_n(K)}{\delta^\frac{2}{n+1}} = d_n\, \Omega^R_p(K).
\]
\end{corollary}
\vskip 3mm
\noindent
It cannot happen for an $R$-ball convex body $K$ that one of the $\kappa_i=0$. 
In fact,  for an $R$-convex body $\kappa_i \geq \frac{1}{R}$ for all $i$. This holds as  at every boundary point of $K$ a ball with radius $R$ containing $K$  touches
$\partial K$  and therefore all principal radii of curvature are smaller than or equal to $R$.
It can happen for an $R$-ball convex body that one or all of the $\kappa_i(x) =\frac{1}{R}$.
If one or all of the $\kappa_i(x)=\frac{1}{R}$ almost everywhere on $\partial K$, then $as^R(K)=0$ and $\Omega_p^R(K)=0$.
For instance this is the case for {\em $R$-ball polyhedra},  i.e. the intersection of
finitely many $R$-balls, see e.g., \cite{BezdekLangiNazodi}. For  $R$-ball polyhedra all principal curvatures are  equal to $\frac{1}{R}$ a.e. 
In particular $as^R(R B^n_2)=0$ and $\Omega_p^R(R B^n_2)=0$.
\vskip 2mm
\noindent
Like the classical $L_p$ affine surface areas, the relative surface areas are valuations and have the same homogeneity properties.
This was shown in \cite{SWY, WernerYalikun}. We show now that, like for the classical surface area,
an isoperimetric inequality also holds for the relative surface area.
\vskip 2mm
\noindent
\begin{proposition}\label{inequality} 
Let $K$ be an $R$-ball convex body in $\mathbb {R}^2$. Let $r=\left(\frac{\vol_2
(K)}{\vol_2(B^2_2)}\right)^\frac{1}{2}$. Then
\begin{equation}\label{equation}
as^R(K) \leq as^R(r B^2_2),
\end{equation}
with equality iff $K= r B^2_2$.
\end{proposition}
\vskip 2mm
\noindent
{\bf Remarks.}
\par
\noindent
(i) A Euclidean ball $r\, B^n_2$ is $R$-ball convex iff $r \leq R$. By assumption on $r$ and as $K$ is $R$-ball convex,   $r^2 = \frac{\vol_2(K)}{vol_2(B^2_2) }\leq \frac{\vol_2(R \, B^2_2)}{\vol_2(B^2_2)}=R^2$.
Thus it holds that  $r \, B^2_2$ is $R$-ball convex.
\vskip 3mm
\noindent
(ii) From the homogeneity property of  relative surface area shown in \cite{SWY} it follows  that  $as^R(r B^2_2)= r^\frac{2}{3} as^\frac{R}{r}( B^2_2)$. Therefore  (\ref{equation}) is equivalent to
$$
\frac{as^R(K)}{\vol_2(K)^\frac{1}{3}}  \leq \frac{as^\frac{R}{r}(B^2_2)}{\vol_2(B^2_2)^\frac{1}{3}},
$$
with equality iff $K= r B^2_2$.
\vskip 3mm
\noindent
The next proposition shows that the $R$-ball illumination body of an $R$-ball convex body is $\tilde{R}$-ball convex with $\tilde{R} \geq R$.
We show this in the plane in the smooth case. The higher dimension general result will be presented in a forthcoming work. 
\par
\noindent
We will use the following notation. For $\theta \in [0, 2 \pi]$, let $h_K(\theta)$ be the support function of $K$, written as a $2 \pi$-periodic function of the outer normal angle $\theta$
and let $\rho_K(\theta)$ be the radius of curvature, i.e., $\rho_K(\theta)=1/ \kappa(\theta, K)$.
Observe that a planar $C^2_+$ convex body is $R$-ball convex iff $\rho_K(\theta) \leq R$ for all $\theta$, see \cite{BezdekLangiNazodi26}.

\begin{proposition}\label{R-convex} 
Let $K$ be an $R$-ball convex body in $\mathbb {R}^2$ that is $C^2_+$. Then we have for $\delta$ small enough
\begin{equation*}
\rho_{I_R^{\delta}(K)}(\theta)= \rho_K(\theta) + d_2 \delta^\frac{2}{3} \left[ \left(\frac{1}{\rho_K(\theta)} -\frac{1}{R}\right)^\frac{1}{3} + \left(\left(\frac{1}{\rho_K(\theta)} -\frac{1}{R}\right)^\frac{1}{3}\right)^{\prime\prime}\right] + O(\delta^\frac{4}{3}),
\end{equation*}
where $d_2=\frac{3^\frac{2}{3}}{2}$.
Consequently $I_R^{\delta}(K)$ is $\tilde{R}$-ball convex with
\begin{equation*}
\tilde{R} = R + d_2 \delta^\frac{2}{3} \max_{\theta \in [0, 2\pi]} \left[ \left(\frac{1}{\rho_K(\theta)} -\frac{1}{R}\right)^\frac{1}{3} + \left(\left(\frac{1}{\rho_K(\theta)} -\frac{1}{R}\right)^\frac{1}{3}\right)^{\prime\prime}\right] + O(\delta^\frac{4}{3}).
\end{equation*}
\end{proposition}

\vskip 4mm
\section{Proofs}
\vskip 4mm

\subsection{Proof of Proposition \ref{convex}}

The proof of Proposition \ref{convex} was communicated to us by F. Besau \cite{BeweisFlorian}.
\par
\noindent
Let $0<\delta<1$, and let $z\in \mathbb{R}^n\setminus K$, such that
$$\vol_n([z,K]_R\setminus K) = \delta.$$
We call $z$ the illuminating point of $K$. Define
$$V_K^R(z)=\vol_n([z,K]_R).$$
Note that $V_K^R(z)$ is continuous in the interior of its domain. Thus, it is enough to show that $V_K^R$ is convex. It is enough to look at all points $z$ such that $z= (1-\lambda)z' +\lambda z''$, where $\lambda\in(0,1)$ and $z',z''$ are such that $[z',z'']$ does not meet $K$. Let 
$$u = \frac{z''-z'}{||z''-z'||}\in S^{n-1}.$$ 
Let $\bar{K}=\proj_{u^{\perp}}(K)$. Note that
$$\bar{z}=\proj_{u^{\perp}}(z)=\proj_{u^{\perp}}(z')=\proj_{u^{\perp}}(z'').$$
Let $\bar{y}\in[\bar{z},\bar{K}]$. Then there exists $s_0(\bar{y}),s_1(\bar{y})$ such that 
$$(\bar{y}+\mathbb{R}u)\cap[z,K]_R=\bar{y}+[s_0(\bar{y}),s_1(\bar{y}))]u.$$
Then we have
$$V_K^R(z)=\int_{[\bar{z},\bar{K}]}\int_{s_0(\bar{y})}^{s_1(\bar{y})}dsd\bar{y}=\int_{[\bar{z},\bar{K}]}(s_1(\bar{y})-s_0(\bar{y}))d\bar{y}.$$
Note for any $y=\bar{y}+su\in[s_0(\bar{y}),s_1(\bar{y})]u$, we can choose $x\in [z,K]_R$ so that 
$$y= (1-\eta)z+\eta x$$
for some $\eta\in(0,1)$. So we have
$$y' = (1-\eta)z' + \eta x$$
$$y''= (1-\eta)z'' + \eta x,$$
where $y'\in [z',K]_R\cap(\bar{y}+\mathbb{R}u)=\bar{y}+[s_0'(\bar{y}),s_1'(\bar{y})]u,$
and $$y''\in [z'',K]_R\cap (\bar{y}+\mathbb{R}u)=\bar{y}+[s_0''(\bar{y}),s_1''(\bar{y})]u.$$
Hence we have
$$[s_0(\bar{y}),s_1(\bar{y})]\subseteq (1-\lambda)[s_0'(\bar{y}),s_1'(\bar{y})]+\lambda[s_0''(\bar{y}),s_1''(\bar{y})]=\left[(1-\lambda)s_0'(\bar{y})+\lambda s_0''(\bar{y}),(1-\lambda)s_1'(\bar{y})+\lambda s_1''(\bar{y})\right].$$
Hence,
$$s_1(\bar{y})-s_0(\bar{y}) \le (1-\lambda)s_1'(\bar{y})+\lambda s_1''(\bar{y})-(1-\lambda)s_0'(\bar{y})-\lambda s_0''(\bar{y})$$
$$= (1-\lambda)[s_1'(\bar{y})-s_0'(\bar{y})]+\lambda[s_1''(\bar{y})-s_0''(\bar{y})]$$
Therefore,
$$V_K^R(z)=\int_{[\bar{z},\bar{K}]}(s_1(\bar{y})-s_0(\bar{y}))d\bar{y}$$
$$\le (1-\lambda)\int_{[\bar{z},\bar{K}]}s_1'(\bar{y})-s_0'(\bar{y})d\bar{y} + \lambda \int_{[\bar{z},\bar{K}]}s_1''(\bar{y})-s_0''(\bar{y})d\bar{y}$$
$$=(1-\lambda)V_K^R(z')+\lambda V_K^R(z'').$$

\vskip 4mm
\subsection{Proof of Proposition \ref{inequality}}
We use Steiner symmetrization. In dimension $2$, the class of ball-convex bodies is invariant under Steiner symmetrization.
This is not the case in dimensions greater than or equal $3$. An example is given in \cite{ArtsteinFlorentin}.
\begin{proof}
\noindent
We first write the setting for dimension $n$ and then specify to $n=2$.
We follow, in part,  the proof in \cite{Hug}. We assume without loss of generality that $0 \in \text{int} (K)$.
Let $\xi \in S^{n-1}$, let $p_\xi(K)$ be the projection of $K$ to $\xi^\perp$, let $K_\xi$ be the relative interior of $p_\xi(K)$ and for $x \in K_\xi$ put
\begin{eqnarray*}
&&f^-(x) = \min\{ \lambda \in \mathbb{R}: x+\lambda \xi \in K\}\\
&&f^+(x) = \max\{ \lambda \in \mathbb{R}: x+\lambda \xi \in K\}.
\end{eqnarray*}
We also put 
$$
K^- = \text{graph} (f^-), \hskip 4mm K^+= \text{graph} (f^+).
$$
Then, with $\text{relbd} (p_\xi(K))$ denoting the relative boundary of $p_\xi(K)$, 
$$
\partial K = K^-  \cup K^+ \cup (\partial K \cap \text{relbd} (p_\xi(K))).
$$
As $\partial K \cap \text{relbd} (p_\xi(K))$ has measure $0$, we have for $x \in K_\xi$, $z^-=(x, f^-(x))$, $z^+=(x, f^+(x))$, 
\begin{eqnarray} \label{2Integrale}
as^R(K) &=& \int_{K^- } \prod_{i=1}^{n-1} \left(\kappa_i(K,z^-) -\frac {1}{R}\right)^\frac{1}{n+1} d \mu_K(z^-) \nonumber \\
&+&
 \int_{K^+ } \prod_{i=1}^{n-1} \left(\kappa_i(K,z^+) -\frac {1}{R}\right)^\frac{1}{n+1} d \mu_K(z^+).
\end{eqnarray}
Let
$
F^- : K_\xi \to \mathbb{R}^n, \hskip 4mm x \to x +f^-(x) \xi.
$
Then
\begin{eqnarray*}
\int_{K^- } \prod_{i=1}^{n-1} \left(\kappa_i(K,z^-) -\frac {1}{R}\right)^\frac{1}{n+1} d \mu_K(z^-) = 
\int_{K_u } \prod_{i=1}^{n-1} \left(\kappa_i(K,F^-(x)) -\frac {1}{R}\right)^\frac{1}{n+1} \left( 1+ \| \nabla f^-(x)\|^2\right)^\frac{1}{2} dx.
\end{eqnarray*}
We have for the Gauss curvature $\kappa$ (see e.g. \cite{LSW2017}),
$$\kappa(K,F^-(x)) = \prod_{i=1}^{n-1}\kappa_i(K,F^-(x))  = \frac{|\det \left[\nabla^2 f^-(x)\right]|}{(1 + \|\nabla f^-(x)\|^2)^\frac{n+1}{2}}
$$
The case  $n=2$.
\begin{eqnarray*}
\left(\kappa(K,F^-(x)) -\frac {1}{R}\right)^\frac{1}{3} &=&
\left(\frac{|(f^-)^{\prime \prime}|}{(1 + (f^{- \prime})^2)^\frac{3}{2}} - \frac{1}{R} \right)^\frac{1}{3}
=\frac{\left(|(f^-)^{\prime \prime}(x)| - (1 +  (f^{- \prime})^2)^\frac{3}{2}\frac{1}{R} \right)^\frac{1}{3}}
{(1 +  (f^{- \prime})^2)^\frac{1}{2}}\\
\end{eqnarray*}
We identify $K_\xi = \mathbb{R}$ and put $S_\xi(K) = S(K)$ where $S_\xi(K)$ is the Steiner symmetrization of $K$ at $\xi^\perp$. 
As repeated Steiner symmetrizations lead to the Euclidean ball $r\, B^2_2$, it is enough
to show that $as^R(K)  \leq as^R(S(K))$.
As   $g=f^-$ and $h=(-f^+)$ are concave, we thus need to show that  
\begin{eqnarray*} 
& \int_{\mathbb{R} } \left[-g^{\prime \prime}- \frac{1}{R} \left(1 +(g^{\prime})^2\right)^\frac{3}{2}
\right]^\frac{1}{3} + \left[-h^{\prime \prime} - \frac{1}{R} \left(1 + (h^{\prime})^2\right)^\frac{3}{2} \right]^\frac{1}{3} dx \\
&\leq
2^\frac{2}{3} \int_\mathbb{R}  \left[ -g^{\prime \prime}- h^{\prime \prime} -\frac{2}{R} \left( 1+\frac{1}{4}(g^\prime+h^\prime)^2\right)^\frac{3}{2}\right]^\frac{1}{3} dx.
\end{eqnarray*}
It is enough to show that point-wise, 
\begin{eqnarray}\label{Gleichung1}
&2^\frac{2}{3} \left[ -g^{\prime \prime}(x) - h^{\prime \prime}(x)  -\frac{2}{R} \left( 1+\frac{1}{4}(g^\prime(x) +h^\prime (x))^2 \right)^\frac{3}{2}\right]^\frac{1}{3} \geq  \nonumber\\
&\left[-g^{\prime \prime}(x)  - \frac{1}{R} \left(1 + (g^{\prime}(x))^2 \right)^\frac{3}{2}
\right]^\frac{1}{3} + \left[-h^{\prime \prime}(x) - \frac{1}{R} \left(1 + (h^{\prime}(x))^2\right)^\frac{3}{2} \right]^\frac{1}{3}.
\end{eqnarray}
We again omit writing the argument and put 
$$ \left[-g^{\prime \prime} - \frac{1}{R} \left(1 + (g^{\prime})^2\right)^\frac{3}{2}\right]^\frac{1}{3} =a,  \hskip 10mm 
 \left[-h^{\prime \prime} - \frac{1}{R} \left(1 + (h^{\prime})^2\right)^\frac{3}{2} \right]^\frac{1}{3}=b.$$
Then showing  (\ref{Gleichung1}) is equivalent to showing 
\begin{eqnarray*}\label{Gleichung2}
&2^\frac{2}{3} \left[ a^3 +b^3 +\frac{1}{R} \left((1 + ( g^{\prime})^2)^\frac{3}{2} + (1 +  (h^{\prime})^2)^\frac{3}{2}  - 2 \left( 1+\frac{1}{4}(g^\prime+h^\prime)^2\right)^\frac{3}{2}\right) \right]^\frac{1}{3} \geq  
a+b.
\end{eqnarray*}
We put $(g^\prime)^2 =u^2$ and $(h^\prime)^2 =v^2$ and observe that
$$
(1 + u^2)^\frac{3}{2} + (1 +  v^2)^\frac{3}{2}  - 2 \left( 1+\frac{1}{4}(u+v)^2\right)^\frac{3}{2} \geq 0,
$$
with equality iff $u=v$.
Thus,  to show  (\ref{Gleichung1}) it is enough to show that
$$
2^\frac{2}{3} \left[ a^3 +b^3 \right]^\frac{1}{3} \geq a+b
$$
which holds with equality if and only if $a=b$.
\vskip 2mm
\noindent
Now we treat the equality characterization. It is clear that we have equality when $K=r\, B^2_2$.
\par
\noindent
Assume now that equality holds. Then we have equality in every Steiner symmetrization. 
In particular we have with the above notations
\begin{eqnarray} \label{F-integral}
 \int_\mathbb{R} F_{g,h}  dx = 0, 
\end{eqnarray}
where we have put
\begin{eqnarray*}
&&\hskip -10mm F_{g,h} = 
2^\frac{2}{3}   \left[ -g^{\prime \prime}- h^{\prime \prime} -\frac{2}{R} \left( 1+\frac{1}{4}(g^\prime+h^\prime)^2\right)^\frac{3}{2}\right]^\frac{1}{3} \\
&& \hskip 10mm - 
  \left[-g^{\prime \prime} - \frac{1}{R} \left(1 +(g^{\prime})^2\right)^\frac{3}{2}
\right]^\frac{1}{3} - \left[-h^{\prime \prime} - \frac{1}{R} \left(1 + (h^{\prime})^2\right)^\frac{3}{2} \right]^\frac{1}{3}
\end{eqnarray*}
and $g$ and $h$ are as above.
As noted above, $F_{g,h}  \geq 0$ everywhere. Hence it follows from (\ref{F-integral}) that $F_{g,h} = 0$ almost everywhere. 
As $K$ is $R$-ball convex, $K$ is $C^2_+$ almost everywhere and hence the functions $g$ and $h$ are continuous. Therefore $F_{g,h} = 0$  everywhere. 
This means that $f^- = -f^+$ and $K$ is symmetric about $\xi^\perp$ and consequently symmetric about every hyperplane and thus a Euclidean ball.  
\end{proof}

\vskip 4mm
\subsection{Proof of Theorem \ref{theorem:limit}}

We need several more ingredients.  The first can be found in e.g., \cite{SWY}.
\noindent
\vskip 3mm
\begin {proposition} \label{integral}
For $1 \leq i \leq n$, let $c_i >0$. Then
$$
\int_{S^{n-1}} \frac{d\sigma(\xi)}
{\left(\sum_{i=1}^{n} c_i \xi_i^2\right)^\frac{n}{2}} = \frac{2 \pi^\frac{n}{2}}{\Gamma(\frac{n}{2})} \left[\prod_{i=1}^{n} c_i ^\frac{1}{2}\right]^{-1} = \sigma (S^{n-1}) \left[\prod_{i=1}^{n} c_i ^\frac{1}{2}\right]^{-1}.
$$
\end{proposition}
\vskip 2mm
\noindent
We require in Proposition \ref{integral} that all $c_{i}>0$, $1\leq i\leq n$. If only one of the $c_{i}$
equals $0$ the integral equals $\infty$. Indeed, suppose that $c_{n}=0$, while the others are strictly greater than $0$. Then on a set of volume $d(n)\,\varepsilon^{n-1}$ the integrand $\left(\sum_{i=1}^{n} c_i \xi_i^2\right)^\frac{n}{2}$ is greater than $\frac{1}{\varepsilon^{n}}\left(\sum_{i=1}^{n} c_i\right)^\frac{n}{2}$. Thus the integral is unbounded.
\vskip 3mm
\noindent
The next lemma is also standard, see e.g., \cite{Werner94}.   
\begin {lemma} \label{vol-diff}
Let $K$ and $L$  be  convex bodies in $\mathbb{R}^n$ such that $0 \in \text{int} (K) \subset L$. Then
$$
\vol_n(L)-\vol_n(K) = \frac{1}{n} \int_{\partial K} \langle x, N_K(x) \rangle \left[\left(\frac{\|x_L\|}{\|x\|}\right)^n -1 \right] \, d \mu_K(x), 
$$
where $x\in\partial K$, $x_L=\{\alpha x:\alpha\geq 0 \}\cap \partial L$.
\end{lemma} 
\vskip 2mm
\begin{remark} 
The lemma also holds if $L$ is star shaped, e.g., \cite{WY:2008}.
\end{remark}
\vskip 3mm
\noindent
We can assume without loss of generality that $0 \in \text{int} (K)$. 
We put for $x\in\partial K$, 
\begin{equation}\label{xdelta}
x^\delta =\{\alpha x:\alpha\geq 0 \}\cap \partial (I_R^\delta(K)) \hskip 4mm \text{respectively} \hskip 4mm x^\delta =\{\alpha x:\alpha\geq 0 \}\cap \partial (I_{R, \varphi} ^\delta(K)),
 \end{equation}
depending if we deal with the $R$ -illumination body respectively the weighted  $R$ -illumination body.
We  apply the lemma for the  $R$-illumination bodies.  This means that 
\begin{equation}\label{proofstep1}
\lim_{\delta \to 0} \frac{\vol_n\left(I_R^\delta(K)\right)-\vol_n(K) } {\delta^\frac{2}{n+1}} = \frac{1}{n} \lim_{\delta \to 0} \int_{\partial K} \frac{\langle x, N_K(x) \rangle}{\delta^\frac{2}{n+1}} \left[\left(\frac{\|x^\delta\|}{\|x\|} \right)^n-1 \right] \, d \mu_K(x). 
\end{equation}
\vskip 3mm
\noindent
For  the next lemma, we recall the \emph{rolling function} $r_K:\partial K \to [0,\infty)$ of a convex body $K$, which  was introduced by McMullen in \cite{McMullen}, see also \cite{SW:1990}.
For $x \in \partial K$ with unique outer normal $N_K(x)$ it  is defined by 
\begin{equation*}
	r_K(x) = \max\{\rho: B^{n}_2 (x - \rho N_{K} (x), \rho ) \subset K\}, 
\end{equation*}
i.e., $r_K(x)$ is the maximal radius of a Euclidean ball inside $K$ that contains $x$.
If $N_K(x)$ is not unique, $r_K(x) =0$. By McMullen \cite{McMullen} (also \cite{SW:1990}) $r_K(x)>0$  almost everywhere on $\partial K$.
It was shown in \cite{SW:1990} that for all $0 \leq \alpha <1$, 
\begin{equation}\label{r}
\int_{\partial K} \frac{1}{r_K(x)^\alpha} d \mu_K(x) < \infty.
\end{equation}
\vskip 3mm
\noindent
The next lemma and its proof is the analog of Lemma 2 of \cite{Werner94}. \vskip 3mm
\begin{lemma} \label{bounded}
Let $K$ be an $R$-ball convex body in $\mathbb R^{n}$  that contains $0$ as an interior point.
Let $x \in \partial K$ such that $r_{K}(x)>0$. 
Then there is $\delta_0$ such that for all $\delta \leq \delta_0$,  
\begin{equation*}\label{bounded-1}
\frac{1}{n} \frac{\langle x, N_K(x) \rangle}{\delta^\frac{2}{n+1}} \left[\left(\frac{\|x^\delta\|}{\|x\|} \right)^n-1 \right]\leq \gamma_n \, r_{K}(x)^{-\frac{n-1}{n+1}},
\end{equation*}
where $\gamma_n$ depends on $n$ and $K$ only.
\end{lemma}
\vskip 2mm
\begin{proof}
Let $x\in\partial K$ such that $r_K(x)>0$. Let $x^\delta$ be as in (\ref{xdelta}) and let $\tilde{x}^\delta= \{\alpha x:\alpha\geq 0 \}\cap \partial K^\delta$.
By (\ref{inclusion}) we have that $\|x^\delta\| \leq \|\tilde{x}^\delta\|$ and hence
$$
\frac{1}{n} \frac{\langle x, N_K(x) \rangle }{\delta^\frac{2}{n+1}} \left[\left(\frac{\|x^\delta\|}{\|x\|} \right)^n -1 \right]\leq \frac{1}{n} \frac{\langle x, N_K(x) \rangle }{\delta^\frac{2}{n+1}} \left[\left(\frac{\|\tilde{x}^\delta\|}{\|x\|}\right)^n -1 \right]\leq \gamma_n \, r_{K}(x)^{-\frac{n-1}{n+1}}.
$$
The last inequality follows from  Lemma 2 of  \cite{Werner94}. This finishes  the proof.    
    \end{proof}
\vskip 3mm
\noindent
By Lemma \ref{bounded} and Lebesgue Dominated Convergence Theorem, we can interchange integration and limit in (\ref{proofstep1}) and get 
\begin{equation}\label{proofstep2}
\lim_{\delta \to 0} \frac{\vol_n\left(I_R^\delta(K)\right)-\vol_n(K) } {\delta^\frac{2}{n+1}} = \frac{1}{n}  \int_{\partial K} \lim_{\delta \to 0} \frac{\langle x, N_K(x) \rangle}{\delta^\frac{2}{n+1}} \left[\left(\frac{\|x^\delta\|}{\|x\|}\right)^n -1 \right] \, d \mu_K(x). 
\end{equation}
As $x$ and $x^\delta$ are colinear, we get that $\|x^\delta\|= \|x\| + \|x^\delta-x\|$ and thus
\begin{eqnarray*}\label{proofstep3}
\frac{\langle \frac{x}{\|x\|}, N_K(x) \rangle}{\delta^\frac{2}{n+1}} \|x^\delta-x\| &\leq&\frac{1}{n} \frac{\langle x, N_K(x) \rangle}{\delta^\frac{2}{n+1}} \left[\left(\frac{\|x^\delta\|}{\|x\|}\right)^n -1 \right] \nonumber \\
& \leq &
\frac{\langle \frac{x}{\|x\|}, N_K(x) \rangle}{\delta^\frac{2}{n+1}} \|x^\delta-x\| \left(1+ \beta_n \frac{\|x^\delta-x\|}{\|x\|}\right),
\end{eqnarray*}
where $\beta_n$ is a constant. Thus 
\begin{eqnarray}\label{proofstep3}
&&\hskip -15mm  \int_{\partial K} \lim_{\delta \to 0} \frac{\langle \frac{x}{\|x\|}, N_K(x) \rangle}{\delta^\frac{2}{n+1}} \|x^\delta-x\| d \mu_K(x) \leq \lim_{\delta \to 0} \frac{\vol_n\left(I_R^\delta(K)\right)-\vol_n(K) } {\delta^\frac{2}{n+1}} \nonumber \\
 &&\hskip 10mm \leq \int_{\partial K} \lim_{\delta \to 0} \frac{\langle \frac{x}{\|x\|}, N_K(x) \rangle}{\delta^\frac{2}{n+1}} \|x^\delta-x\|  \left( 1 + \beta_n \frac{ \|x^\delta-x\| }{\|x\|}\right)d \mu_K(x).
\end{eqnarray}
\vskip 3mm
\noindent
We now compute the limit under the integral.
\begin {lemma} \label{limit} 
Let $K$ be an $R$-ball convex body in $\mathbb R^{n}$.
Let $x \in \partial K$ be such that $\kappa_i(x, K) > \frac{1
}{R}$ for all $i=1,\dots,n-1$. Then
$$
\lim_{\delta \to 0} \left\langle \frac{x}{\|x\|}, N_K(x) \right\rangle  \frac{ \|x-x^\delta\|}{\delta^\frac{2}{n+1}} =  
\frac{1}{2} \left [ \frac{n(n+1)}{\vol_{n-1}(B^{n-1}_2)}\right]^\frac{2}{n+1} \left(\prod_{i=1}^{n-1} \left(\kappa_i(x,K) -\frac{1}{R}\right)\right)^{\frac{1}{n+1}}.
$$
\end{lemma}
\vskip 3mm
\noindent
As $x^\delta \to x$ as $\delta \to 0$,  it follows from (\ref{proofstep3}), Lemma \ref{limit} and Proposition \ref{integral}, that 
\begin{eqnarray}\label{proofstep4}
&&\lim_{\delta \to 0}\frac{\vol_n\left(I_R^\delta(K)\right)-\vol_n(K) } {\delta^\frac{2}{n+1}} \nonumber \\
&&=\frac{\left(n(n^{2}-1)\right)^\frac{2}{n+1} }{2} \int_{\partial K} 
\left[\int_{S^{n-2}} \frac{d\sigma(\xi)}
{\left (\sum_{i=1}^{n-1} \left(\kappa_i(x, K) -\frac{1}{R} \right) \xi_i^2\right)^\frac{n-1}{2}}\right]^{-\frac{2}{n+1}} d \mu_K(x) \nonumber \\
&=&\frac{\left(n(n^{2}-1)\right)^\frac{2}{n+1} }{2  \left(\sigma(S^{n-1})\right)^\frac{2}{n+1}} \int_{\partial K} \prod _{i=1}^{n-1} \left( \kappa_i (K, x) -\frac{1}{R} \right)^\frac{1}{n+1}d \mu_K(x).
\end{eqnarray}
This finishes the proof of Theorem \ref{theorem:limit}. It only remains to prove Lemma \ref{limit}.
\vskip 4mm
\noindent
\subsubsection{Proof of Lemma \ref{limit}}

For the proof of Lemma \ref{limit}, we need more ingredients (see e.g., \cite{SW4}).
\par
\noindent
Let $K$ be a convex body in $\mathbb{R}^n$ 
such that $0\in\partial K$ and such that the indicatrix of Dupin in $0$ is
an ellipsoid. For a convex body $K$ in $\mathcal {K}_R^+$ this holds everywhere on $\partial K$. 
We can assume that $N_{K}(0)=-e_{n}$.
Let  $\sum_{i=1}^{n-1}\frac{\xi_{i}^{2}}{b_{i}^{2}}=1$ be the equation of the indicatrix at $0$.
Then the principal Gauss-Kronecker curvatures are $b_{i}^{-2}$,
$i=1,\dots,n-1$ and the Gauss-Kronecker curvature of $K$ at $0$ is $\prod_{i=1}^{n-1}b_{i}^{-2}$.
Let $\mathcal E$ be  the  standard approximating ellipsoid, where
\begin{equation}\label{StandardEll1}
\mathcal E
=\left\{\xi\in\mathbb R^{n}\left|
\sum_{i=1}^{n-1}\frac{\xi_{i}^{2}}{b_{i}^{2}}+
\frac{\left(\xi_{n}-\left(\prod_{i=1}^{n-1}b_{i}\right)^{\frac{2}{n-1}}\right)^{2}}
{\left(\prod_{i=1}^{n-1}b_{i}\right)^{\frac{2}{n-1}}}
\leq\left(\prod_{i=1}^{n-1}b_{i}\right)^{\frac{2}{n-1}}
\right.\right\}.
\end{equation}
We put $b_{i}=\frac{a_{i}}{\sqrt{a_{n}}}$ and $a_{n}=\left(\prod_{i=1}^{n-1}b_{i}\right)^{\frac{2}{n-1}}$
and we get for \eqref{StandardEll1}
\begin{equation}\label{StandardEll2}
\mathcal E
=\left\{\xi\in\mathbb R^{n}\left|
\sum_{i=1}^{n-1}\frac{a_{n}}{a_{i}^{2}}\xi_{i}^{2}+
\frac{\left(\xi_{n}-a_{n}\right)^{2}}
{a_{n}}
\leq a_{n}
\right.\right\}
\end{equation}
and in particular for $1 \leq i \leq n-1$
\begin{equation}\label{pc}
\kappa_i(x, K) = \frac{a_n}{a_i^2}.
\end{equation}
We write
$$
\xi_{n}=a_{n}-a_{n}\sqrt{1-\sum_{i=1}^{n-1}\frac{1}{a_{i}^{2}}\xi_{i}^{2}}.
$$
We will need  two further ellipsoids $\mathcal{E} (\varepsilon^-)$
and $\mathcal{E} (\varepsilon^+)$, one  slightly smaller than $\mathcal E$
and the other slightly bigger.
\begin{equation}\label{StandardEll4}
\mathcal E(\varepsilon^-)
=\left\{\xi\in\mathbb R^{n}\left|
\sum_{i=1}^{n-1}\frac{a_{n}}{(1-\varepsilon)^{2}a_{i}^{2}}\xi_{i}^{2}+
\frac{\left(\xi_{n}-a_{n}\right)^{2}}
{a_{n}}
\leq a_{n}
\right.\right\}.
\end{equation}
\begin{equation}\label{StandardEll5}
\mathcal E(\varepsilon^+)
=\left\{\xi\in\mathbb R^{n}\left|
\sum_{i=1}^{n-1}\frac{a_{n}}{(1+\varepsilon)^{2}a_{i}^{2}}\xi_{i}^{2}+
\frac{\left(\xi_{n}-a_{n}\right)^{2}}
{a_{n}}
\leq a_{n}
\right.\right\}.
\end{equation}
\vskip4mm
\noindent
We will first prove a lemma for ellipsoids.
\begin{lemma} \label{ellipsoid}
Let $\mathcal {E}$ be an $R$ ball convex ellipsoid  with equation $\sum_{i=1}^{n-1}\frac{x_i^2}{a_i^2} + \frac{(x_n-a_n)^2}{a_n^2}=1$.
Let $\varepsilon >0$ be given.  Let   $x^\delta=(x_1^\delta, \dots, x_{n-1}^\delta, -x_n^\delta) \in \partial(I_{R}^\delta (\mathcal{E}))$, with $x_n^\delta >0$.
Then we have for  small enough $\delta$,    
with constants $c$ and $d$, 
\begin{eqnarray*}
&&\hskip -10 mm \frac{2^\frac{n+1}{2} \vol_{n-1}(B^{n-1}_2)}{n(n+1)} \, \frac{(\|x^\delta\| \cos \theta)^\frac{n+1}{2}}
{\left(\prod_{i=1}^{n-1} \left(\frac{a_n}{a_i^2} -\frac{1}{R}\right)\right)^\frac{1}{2}} - d \, \varepsilon \leq 
\vol_n([x^\delta, \mathcal{E}]_R) - \vol_n(\mathcal{E}) \leq  \\
&&\hskip 30mm\frac{2^\frac{n+1}{2} \vol_{n-1}(B^{n-1}_2)}{n(n+1)} \, \frac{(\|x^\delta\| \cos \theta)^\frac{n+1}{2}}
{\left(\prod_{i=1}^{n-1} \left(\frac{a_n}{a_i^2} -\frac{1}{R}\right)\right)^\frac{1}{2}} + c \, \varepsilon,
\end{eqnarray*} 
where $\theta$ is the angle between $-e_n$ and $x^\delta$.
\par
\noindent
In particular, for $x^\delta=(0, \dots, 0 , -x_n^\delta)$, with $x_n^\delta >0$, we get
\begin{eqnarray*}
&&\hskip -10 mm \frac{2^\frac{n+1}{2} \vol_{n-1}(B^{n-1}_2)}{n(n+1)} \, \frac{(x_n^\delta)^\frac{n+1}{2}}
{\left(\prod_{i=1}^{n-1} \left(\frac{a_n}{a_i^2} -\frac{1}{R}\right)\right)^\frac{1}{2}} - d \, \varepsilon \leq 
\vol_n([x^\delta, \mathcal{E}]_R) - \vol_n(\mathcal{E}) \leq  \\
&&\hskip 30mm\frac{2^\frac{n+1}{2} \vol_{n-1}(B^{n-1}_2)}{n(n+1)} \, \frac{(x_n^\delta)^\frac{n+1}{2}}
{\left(\prod_{i=1}^{n-1} \left(\frac{a_n}{a_i^2} -\frac{1}{R}\right)\right)^\frac{1}{2}} + c \, \varepsilon.
\end{eqnarray*} 
\end{lemma}
\vskip 3mm
\begin{proof}
Without loss of generality we can assume that  $R>a_n\ge a_{n-1}\ge...\ge a_1$. 
Write $\tilde {x}= (x_1,x_2,...,x_{n-1})$ for a vector in $\mathbb{R}^{n-1}$ and 
$$
f(\tilde {x})=f(x_1,x_2,...,x_{n-1}) =a_n-a_n\left(1-\sum_{i=1}^{n-1}\frac{x_i^2}{a_i^2}\right)^{1/2}
$$
for  the equation of  the ellipsoid.  
For $k \in \mathbb{R}_+$, let $z^k=(z_1^k, z_2^k,...,z_{n-1}^k, z_n^k)$, $z_n^k \geq 0$,  and let $z^k+R\, B_2^n$ be the Euclidean  ball with center $z^k$ and radius $R$ 
such that $x^\delta \in \partial (z^k+R\, B_2^n)$.
Then this ball has equation 
$$
g_k(\tilde {x})=g_k(x_1,...,x_{n-1})=z^k_n-R\left(1-\sum_{i=1}^{n-1}\frac{(x_i-z^k_i)^2}{R^2}\right)^{1/2}
$$
with $z^k_n = R\left(1-\sum_{i=1}^{n-1}\frac{(x_i^\delta-z^k_i)^2}{R^2}\right)^{1/2}-x_n^\delta$. 
For $\delta$ small, the coordinates of $\tilde{x} = (x_1, \dots, x_{n-1})$ and $\tilde{z}^k = (z^k_1, \dots,  z^k_{n-1})$ are small  and we expand $f$ and $g$ around $0$
$$
f(\tilde {x})= \frac{1}{2} \sum_{i=1}^{n-1}\frac{a_n}{a_i^2} \, x_i^2 + O(\|\tilde{x}\|^4)
$$
and
$$
z^k_n = R- \frac{1}{2R} \sum_{i=1}^{n-1}(x_i^\delta-z^k_i)^2 -x_n^\delta +  O((\|\tilde{x}\|^2 + \|\tilde{z}^k\|^2)^2)
$$
and thus 
$$
g_k(\tilde {x})= -x_n^\delta + \frac{1}{2R } \sum_{i=1}^{n-1} x_i^2-{x_i^\delta} ^2- 2 \, z^k_i (x_i-x_i^\delta)  + O((\|\tilde{x}\|^2 + \|\tilde{z}^k\|^2)^2).
$$
For the point $u^k=(u_1^k, \dots, u_n^k)= (\tilde{u}^k, u_n^k)$ where $\mathcal{E}$ and $ z^k+R\, B^n_2$ touch, we have that $\nabla f (\tilde{u}^k) = \nabla g(\tilde{u}^k)$ and thus we get for all $1 \leq i \leq n-1$,
\begin{equation}\label{zed}
z^k_i= - R \, u^k_i \, \left(\frac{a_n}{a_i^2} -\frac{1}{R}\right) + O\left((\|\tilde{x}\|^2 + \|\tilde{z}^k\|^2)  \right).
\end{equation}
As $f(\tilde{u}^k)= g_k(\tilde{u}^k)$, we get
\begin{equation*} \label{ellipseU1}
2 x_n^\delta  + \sum _{i=1}^{n-1}   \frac{a_n}{a_i^2} {x_i^\delta}^2=  \sum _{i=1}^{n-1} \left( \frac{a_n}{a_i^2} -\frac{1}{R}\right) (u_i^k -x_i^\delta)^2 + O\left((\|\tilde{x}\|^2 + \|\tilde{z}^k\|^2)^2\right),
\end{equation*}
or, equivalently,
\begin{equation} \label{ellipseU1}
1=  \sum _{i=1}^{n-1} \left( \frac{a_n}{a_i^2} -\frac{1}{R}\right) \frac{ (u_i^k -x_i^\delta)^2}{2 x_n^\delta  + \sum _{i=1}^{n-1}   \frac{a_n}{a_i^2} {x_i^\delta}^2} + O\left((\|\tilde{x}\|^2 + \|\tilde{z}^k\|^2)^2\right).
\end{equation}
Note  that $2 x_n^\delta  + \sum _{i=1}^{n-1}   \frac{a_n}{a_i^2} {x_i^\delta}^2 >0$ as $x^\delta \notin \mathcal {E}$.
We put 
$$
\beta_i= \left(\frac{2 x_n^\delta  + \sum _{i=1}^{n-1}   \frac{a_n}{a_i^2} {x_i^\delta}^2}{\frac{a_n}{a_i^2} -\frac{1}{R}} \right)^\frac{1}{2}
$$ 
and  consider the $(n-1)$-dimensional ellipsoid $\mathcal{U}_k$ with equation
\begin{equation*} \label{ellipseU2}
\sum _{i=1}^{n-1} \frac{(u_i^k- x_i^\delta)^2}{\beta_i^2} = 1.
\end{equation*}
Then (\ref{ellipseU1}) says that up to a small error the points where $\mathcal{E}$ and $z^k+R\, B^n_2$ touch lie on the boundary of the ellipsoid $\mathcal {U}_k$.
Let $\varepsilon >0$ be given. We choose $\delta$ so small that for the centers $z^k$ and for $x^\delta$ we have that 
\begin{equation}\label{x-estimate}
\max_{z^k} \|\tilde{z}^k\| ^2< \varepsilon \hskip 5mm \text{and} \hskip 5mm 
\max_x \|\tilde{x}^\delta\|^2 < \varepsilon.
\end{equation}
We define $\mathcal {U}_k(\varepsilon^+)$, resp.   $\mathcal {U}_k(\varepsilon^-)$ to be the ellipsoids with equations
$$
\sum _{i=1}^{n-1} \frac{(u_i ^k - x_i^\delta)^2}{(1+\varepsilon)^2 \, \beta_i^2} = 1, \hskip 5mm  \text{resp.} \hskip 5mm \sum _{i=1}^{n-1} \frac{(u_i ^k - x_i^\delta)^2}{(1-\varepsilon)^2 \, \beta_i^2} = 1.
$$
We shift by $x^\delta$ to get new coordinates $\tilde{w} = \tilde{x} - \tilde{x}^\delta$ and thus
$$
f(\tilde {w})= \frac{1}{2} \sum_{i=1}^{n-1}\frac{a_n}{a_i^2} \, (w_i +x_i^\delta)^2 + c_1\, \varepsilon
$$
and, also using (\ref{zed}), 
$$
g_k(\tilde {w})= -x_n^\delta + \frac{1}{2R } \sum_{i=1}^{n-1} w_i^2 + 2\, R  \, u^k_i \left( \frac{a_n}{a_i^2} -\frac{1}{R}\right)  +  c_2\, \varepsilon.
$$
Next we integrate over $\mathcal {U}_k(\varepsilon^+) - x^\delta$ resp., $\mathcal {U}_k(\varepsilon^-) - x^\delta$.
We pass to polar coordinates $(\xi, l(\xi)) \in S^{n-2} \times [0, \infty)$. For $\xi \in S^{n-2}$, we put 
$$
l_{0}^+(\xi)=l_{0,k}^+(\xi) =\left(\sum_{i=1}^{n-1} \frac{{\xi_i}^2}{(1+\varepsilon)^2 \, \beta_i^2}\right)^{-\frac{1}{2}},  \hskip 10mm l_{0}^+(\xi)=l_{0,k}^-(\xi) =\left(\sum_{i=1}^{n-1} \frac{{\xi_i}^2}{(1-\varepsilon)^2 \, \beta_i^2}\right)^{-\frac{1}{2}},
$$
where we have written in short $\xi_i$ for $\xi_i^k$. Then $ l_0^+(\xi)$ is such that $l_0^+(\xi) \xi \in \partial(\mathcal{U}_k(\varepsilon^+))$ and $ l_0^-(\xi)$ is such that $l_0^-(\xi) \xi \in \partial(\mathcal{U}_k(\varepsilon^-))$,
We write in short $l=l(\xi)$ and  get with constants $c$ that may change from line to line, 
\begin{eqnarray*}
&&\vol_n([x^\delta, \mathcal{E}]_R) - \vol_n(\mathcal{E})  \leq   \int_{S^{n-2}} \int_{l=0}^{l_0^+(\xi)} [f(\xi,l) -g_k(\xi, l)]  \, l^{n-2} \, dl\, d\sigma (\xi) \leq \\
&&\int_{S^{n-2}} \int_{l=0}^{l_0^+(\xi)} \Bigg[  \sum_{i=1}^{n-1}\frac{1}{2} \, \frac{a_n}{a_i^2} \,(l\,  \xi_i +x_i^\delta)^2 + x_n^\delta   - \\
&& \hskip 30mm  \frac{1}{2R} \left(
\sum_{i=1}^{n-1}  l^2 \, \xi_i^2 + 2\, R\, \left(\frac{a_n}{a_i^2} - \frac{1}{R}\right) l_0^+(\xi)\,  l \, \xi_i^2
\right)\Bigg]  \, l^{n-2} \, dl\, d\sigma (\xi) +c \, \varepsilon\\
\end{eqnarray*}
We integrate with respect to the variable $l$ and get
\begin{eqnarray*}
&& \int_{l=0}^{l_0^+(\xi)} \Bigg[  \sum_{i=1}^{n-1}\frac{1}{2} \, \frac{a_n}{a_i^2} \,(l\,  \xi_i +x_i^\delta)^2 + x_n^\delta  - 
\frac{1}{2R} \left(
\sum_{i=1}^{n-1}  l^2 \, \xi_i^2 + 2\, R\, \left(\frac{a_n}{a_i^2} - \frac{1}{R}\right) l_0^+(\xi)\,  l \, \xi_i^2
\right)\Bigg]  \, l^{n-2} \, dl\\
&&= \left(\frac{2 x_n^\delta  + \sum _{i=1}^{n-1}   \frac{a_n}{a_i^2} {x_i^\delta}^2}{\sum _{i=1}^{n-1}  \xi_i^2 \left(\frac{a_n}{a_i^2} -\frac{1}{R}\right)} \right)^\frac{n-1}{2}
\Bigg[ \frac{1}{n(n^2-1)} \left(2 x_n^\delta  + \sum _{i=1}^{n-1}   \frac{a_n}{a_i^2} {x_i^\delta}^2\right) + \\ 
&&\hskip 60mm  \sum _{i=1}^{n-1}   \frac{a_n}{a_i^2}\frac{\xi_i\,  {x_i^\delta}}{n} 
\left(\frac{2 x_n^\delta  + \sum _{i=1}^{n-1}   \frac{a_n}{a_i^2} {x_i^\delta}^2}{\sum _{i=1}^{n-1}  \xi_i^2 \left(\frac{a_n}{a_i^2} -\frac{1}{R}\right)} \right)^\frac{1}{2} \Bigg].
\end{eqnarray*}
Next we integrate over $S^{n-2}$. We notice   that 
$$
\int_{S^{n-2}} \sum _{i=1}^{n-1}   \frac{a_n}{a_i^2}\frac{\xi_i\,  {x_i^\delta}}{n} 
\left(\frac{2 x_n^\delta  + \sum _{i=1}^{n-1}   \frac{a_n}{a_i^2} {x_i^\delta}^2}{\sum _{i=1}^{n-1}  \xi_i^2 \left(\frac{a_n}{a_i^2} -\frac{1}{R}\right)} \right)^\frac{n}{2} =0.
$$
We can assume that  for $\delta$ small enough, with a constant $c>0$, $x_i^\delta \leq c \, x_n^\delta$, for $1 \leq i \leq n-1$, and 
obtain 
\begin{eqnarray*}
&\vol_n([x^\delta, \mathcal{E}]_R) - \vol_n(\mathcal{E}) \leq  \\
&  \frac{2^\frac{n+1}{2}\left( x_n^\delta  + \frac{1}{2} \sum _{i=1}^{n-1}   \frac{a_n}{a_i^2} {x_i^\delta}^2\right)^\frac{n+1}{2}}{n(n+1)}\, 
\vol_{n-1} (B^{n-1}_2) \,   \prod_{i=1}^{n-1}  \left(\frac{a_n}{a_i^2} - \frac{1}{R}\right)^{-\frac{1}{2}} + c \, \varepsilon \leq \\
&\frac{2^\frac{n+1}{2}\, {x_n^\delta}^\frac{n+1}{2}} {n(n+1)}\,  (1 + d\,  \varepsilon ^\frac{1}{2})^\frac{n+1}{2}\, 
\vol_{n-1} (B^{n-1}_2) \,   \prod_{i=1}^{n-1}  \left(\frac{a_n}{a_i^2} - \frac{1}{R}\right)^{-\frac{1}{2}} + c \, \varepsilon.
\end{eqnarray*}
For the second last inequality we have used Proposition \ref{integral}. We have also used (\ref{x-estimate}) and that the absolute constants may change from line to line.
Finally we observe that 
$$
x_n^\delta = \langle x^\delta, e_n\rangle=\|x^\delta\| \cos \theta.
$$
\par
\noindent
The inequality from below is done similarly, integrating over $\mathcal {U}_k(\varepsilon^-)$.
\end{proof}
\vskip 4mm
\noindent
Now we give the proof of Lemma \ref{limit}.
\begin{proof}
Let 
$x \in  \partial K$ with outer normal $N_K(x)$. We put for short $\left\langle \frac{x}{\|x\|}, N_K (x)\right\rangle =\cos \theta$. 
Note that as  $x$ and $x^\delta$ are colinear, we also have 
$$
\left\langle \frac{x}{\|x\|}, N_K (x)\right\rangle = \left\langle \frac{x^\delta}{\|x^\delta\|}, N_K (x)\right\rangle =\cos \theta.
$$
Moreover,  
\begin{equation}\label{cos}
\left\langle \frac{x}{\|x\|}, N_K (x)\right\rangle  \|x - x^\delta\|=\cos \theta  \,  \|x - x^\delta\|.
\end{equation}
Locally around $x$, $\partial K$  can be approximated by an ellipsoid $\mathcal{E}$. We  make this precise:
\newline
It will be convenient, also  for the proof,  to shift $K$ by $x$ and rotate $K$ such  that $x=0$ and such that $N_K(x)=-e_n$. Let $\varepsilon >0$  be given.
Let $\mathcal{E}$ be the ellipsoid (\ref{StandardEll2}) with center at  $a_{n} e_n$ and lengths of its principal axes $a_1, \dots, a_{n}$. 
We can also assume that the principal axes of $\mathcal{E}$ coincide with the basis vectors $e_1, \dots, e_{n-1}$ and that 
\begin{equation}\label{R+a}
R \geq a_n \geq \dots \geq a_1.
\end{equation}
Let $\mathcal{E} (\varepsilon^-)$ be the ellipsoid (\ref{StandardEll4}) centered at $ a_{n} e_n$ 
whose principal axes  coincide with the ones of $\mathcal{E}$,  but have lengths $(1-\varepsilon) a_1, \dots, (1-\varepsilon) a_{n-1}, a_{n}$. Similarly, let $\mathcal{E} (\varepsilon^+)$ be the ellipsoid (\ref{StandardEll5}) 
centered at $a_{n} e_n$, with the same principal axes as $\mathcal{E}$,  but with lengths $(1+\varepsilon) a_1, \dots, (1+\varepsilon) a_{n-1}, a_{n}$.
Then
$$ x=0 \in \partial \mathcal{E}  \hskip 3mm \text { and } \hskip 3mm N_{\mathcal{E}}(x) = N_K(x),$$
and (see, e.g., \cite{SW4})  there exists a $\Delta_\varepsilon >0$ such that 
\begin{eqnarray}\label{ellipse}
 H^-\left( \Delta_\varepsilon e_n, e_n\right)  \  \cap  \  \mathcal{E} (\varepsilon^-)
  \subseteq  H^-\left( \Delta_\varepsilon e_n, e_n\right) \  \cap \   K 
\subseteq H^-\left( \Delta_\varepsilon e_n, e_n\right) \  \cap  \  \mathcal{E} (\varepsilon^+) .
\end{eqnarray}
These inclusions explain why we call $\mathcal{E}$ the approximating ellipsoid to $\partial K$ in $x$.
Let $x^\delta \in \partial I^\delta_R(K)$. 
As $x^\delta \to x$ as $\delta \to 0$ and as $\mathcal{E}$ is the approximating ellipsoid to $\partial K$ in $x$, 
we can  choose $\delta$ so small that  for all support balls $z +R B^n_2$ to $K^R_\delta$  in $x_\delta$ we have
\begin{equation}\label{delta}
[x^\delta,  \mathcal{E} (\varepsilon^-)]_R \setminus  \mathcal{E} (\varepsilon^-) \subseteq   H^-\left( \Delta_\varepsilon e_n, e_n\right),  \hskip 4mm [x^\delta,  \mathcal{E} (\varepsilon^+)]_R \setminus  \mathcal{E} (\varepsilon^+) \subseteq   H^-\left( \Delta_\varepsilon e_n, e_n\right).
 \end{equation}
The relations (\ref{delta}) state that for small enough $\delta$,  all $R$-ball convex hulls with  $x^\delta$
are in this  set where $\mathcal{E}$ is a good approximation for $K$ in the sense of (\ref{ellipse}).
\par
\noindent
Let $H\left(x^\delta, e_n\right)$ be the hyperplane through $x^\delta$ and orthogonal to $e_n$. 
\par
\noindent
Let $-\langle \frac{x}{\|x\|}, N_K (x)\rangle  \|x-x^\delta\|  e_n = - \cos \theta \,  \|x^\delta\| e_n$ be the intersection of this hyperplane with the $-e_n$-axis.
Then 
\begin{eqnarray}
\delta &=& \vol_n([x^\delta, K]_R) - \vol_n(K)  \geq  \vol_n([x^\delta,  \mathcal{E}(\varepsilon^+) ]_R) - \vol_n( \mathcal{E}(\varepsilon^+)) \nonumber\\
&\geq&  \vol_n([- \cos \theta \,  \|x^\delta\| e_n,  \mathcal{E}(\varepsilon^+) ]_R) - \vol_n( \mathcal{E}(\varepsilon^+)).
\end{eqnarray}
The last inequality holds as the illuminated volume is minimal if $x^\delta$ is above $x=0$ in direction of the outer normal.
Shifting back by $x$, 
we get with Lemma \ref{ellipsoid} with $\Delta = \langle \frac{x}{\|x\|}, N_K (x)\rangle  \|x-x^\delta\|$ and small enough $\delta$, 
\begin{eqnarray}\label{nachunten}
\delta &\geq& 
 \frac{(1 - c\,  \varepsilon) \, 2^\frac{n+1}{2} \vol_{n-1}(B^{n-1}_2)}{n+1} \, 
 \left[ \langle \frac{x}{\|x\|}, N_K (x) \rangle \|x-x^\delta\| \right]^\frac{2}{n+1}
\left(\prod_{i=1}^{n-1} \left(\frac{a_n}{a_i^2} -\frac{1}{R}\right)\right)^{-\frac{1}{2}}, 
\end{eqnarray}
or, equivalently,
\begin{eqnarray*}
\frac{\langle \frac{x}{\|x\|}, N_K (x) \rangle \|x-x^\delta\|}{\delta^\frac{2}{n+1}} &\leq&(1 + c\, \varepsilon) \frac{1}{2} \left [ \frac{n+1}{\vol_{n-1}(B^{n-1}_2)}\right]^\frac{2}{n+1} \left(\prod_{i=1}^{n-1} \left(\frac{a_n}{a_i^2} -\frac{1}{R}\right)\right)^{\frac{1}{n+1}}\\
&=&(1 + c\, \varepsilon) \frac{1}{2} \left [ \frac{n+1}{\vol_{n-1}(B^{n-1)}_2)}\right]^\frac{2}{n+1} \left(\prod_{i=1}^{n-1} \left(\kappa_i(x,K) -\frac{1}{R}\right)\right)^{\frac{1}{n+1}}.
\end{eqnarray*}
The last equality follows from (\ref{pc}).
\vskip 2mm
\noindent
Now we treat the estimate from below. We keep the same coordinate setup as above. In particular,  $x=0$ is in $\partial K$ with outer normal $N_K(x)=N_K(0)=-e_n$.
Let $\delta >0$ be so small that (\ref{delta}) holds and let $x^\delta \in \partial I_R^\delta(K)$.
Recall that  $\theta$ is the angle between
$e_n$  and $x^\delta$.   Then
\begin{eqnarray}
\delta &=& \vol_n([x^\delta, K]_R) - \vol_n(K)  \leq  \vol_n([x^\delta,  \mathcal{E}(\varepsilon^-) ]_R) - \vol_n( \mathcal{E}(\varepsilon^-)) \nonumber\\
&\leq&  \vol_n([- \cos \theta \,  \|x^\delta\| e_n,  \mathcal{E}(\varepsilon^-) ]_R) - \vol_n( \mathcal{E}(\varepsilon^-))\\
&\leq& \frac{2^\frac{n+1}{2} \vol_{n-1}(B^{n-1}_2)}{n(n+1)} \, \frac{(\|x^\delta\| \cos \theta)^\frac{n+1}{2}}
{\left(\prod_{i=1}^{n-1} \left(\frac{a_n}{a_i^2} -\frac{1}{R}\right)\right)^\frac{1}{2}} + c \, \varepsilon.
\end{eqnarray}
In the last inequality we have again used Lemma (\ref{ellipsoid}). We continue similarly as above.
\end{proof}
\vskip 4mm
\subsection{Proof of Proposition \ref{R-convex}}
Let $x \in \partial K$ with outer normal $N_K(x)=u$ and let $x^\delta \in \partial I_R^\delta(K)$ be as in (\ref{xdelta}). Then $h_K(u) = \langle \frac{x}{\|x\|}, u \rangle = \langle \frac{x^\delta}{\|x^\delta\|}, u\rangle$. As $x^\delta \to x$ as $\delta \to 0$, we get for small enough $\delta$, with an absolute constant $d$, 
\begin{eqnarray*}\label{}
\langle \frac{x}{\|x\|}, u \rangle \|x^\delta-x\| \left(1- d \|x^\delta-x\| \right) \leq h_{I_{R}^\delta(K)}(u) - h_K(u) \leq \langle \frac{x}{\|x\|}, u \rangle  \|x^\delta-x\|.
\end{eqnarray*}
It now follows from Lemma \ref{limit} that for $\delta$ small enough
\begin{eqnarray*}\label{}
h_{I_{R}^\delta(K)}(u) =h_K(u) +d_n \delta^\frac{2}{n+1} \left(\prod_{i=1}^{n-1} \left(\kappa_i(x,K) -\frac{1}{R}\right)\right)^{\frac{1}{n+1}} + O(\delta^\frac{4}{n+1}).
\end{eqnarray*}
Now we stay in dimension $2$. A $C^2_+$ planar body $K$ is $R$-ball convex iff for all $\theta \in[0,2\pi]$
\[ \rho_K(\theta) = h_K(\theta) - h_K(\theta)^{\prime \prime} \leq R.
\]
Now by above,
\[ \rho_{I_{R}^\delta(K)}(\theta) = h_K(\theta) +d_2 \delta^\frac{2}{3} \left[ \left(\frac{1}{\rho_K(\theta)} - \frac{1}{R}\right)^\frac{1}{3} +  \left(\left(\frac{1}{\rho_K(\theta)} - \frac{1}{R}\right)^\frac{1}{3}\right)^{\prime \prime}\right] +O(\delta^\frac{4}{3}).
\]
Thus $I_{R}^\delta(K)$ is $\tilde{R}$-ball convex with
\[ \tilde{R} = \max_{\theta \in [0,  2 \pi] }\left( \rho_K(\theta) +d_2 \delta^\frac{2}{3} \left[ \left(\frac{1}{\rho_K(\theta)} - \frac{1}{R}\right)^\frac{1}{3} +  \left(\left(\frac{1}{\rho_K(\theta)} - \frac{1}{R}\right)^\frac{1}{3}\right)^{\prime \prime}\right] +O(\delta^\frac{4}{3})\right) .
\]
\vskip 4mm
\subsection{Proof of Theorem \ref{theorem:weighted limit} }
The proof of Theorem \ref{theorem:weighted limit}  follows along the same lines as the one of Theorem \ref{theorem:limit}  with some modification that we outline now.
\vskip 3mm
\noindent
We can assume without loss of generality that $0 \in \text{int} (K)$. 
Again, we put  for $x\in\partial K$, 
\begin{equation}\label{xdelta2}
x^\delta =\{\alpha x:\alpha\geq 0 \}\cap \partial (I_{R, \varphi} ^\delta(K)).
 \end{equation}
We  apply  Lemma \ref{vol-diff}  with $L=I_{R, \varphi}^\delta (K)$,  which is possible by Remark 1.  This means that 
\begin{equation}\label{proofstep1,2}
\lim_{\delta \to 0} \frac{\vol_n\left(I_{R, \varphi} ^\delta(K)\right)-\vol_n(K) } {\delta^\frac{2}{n+1}} = \frac{1}{n} \lim_{\delta \to 0} \int_{\partial K} \frac{\langle x, N_K(x) \rangle}{\delta^\frac{2}{n+1}} \left[\left(\frac{\|x^\delta\|}{\|x\|} \right)^n-1 \right] \, d \mu_K(x). 
\end{equation}
\vskip 2mm
\noindent
Now we  show the analog of Lemma \ref{bounded}, i.e., we show that for $x^\delta \in \partial(I_{R, \varphi}^\delta (K) )$,
$$
\frac{\langle x, N_K(x) \rangle}{\delta^\frac{2}{n+1}} \left[\left(\frac{\|x^\delta\|}{\|x\|} \right)^n-1 \right] 
$$ 
is bounded above by an integrable function, namely a multiple of the rolling function $r(x)^\frac{n-1}{n+1}$.
 As $I_{R, \varphi}^\delta (K) \to K$ in the Hausdorff metric as $\delta \to 0$, there is $\delta_1$  that for all $\delta \leq \delta_1$, 
$I_{R, \varphi}^\delta (K)$ is a compact star body contained in $U$.  Since $\varphi:U\to (0,+\infty)$ is continuous, we thus have that there exist $0 < m(\varphi) \leq M(\varphi) < \infty$ such that for all $y \in I_{R, \varphi}^\delta (K)$,
\begin{equation*}
    m_{\varphi}\leq \varphi(y) \leq M_\varphi.
\end{equation*}
Then
\begin{equation*}
    \delta =  \int_{[x^\delta,K]_R\setminus K} \varphi(y) \, dy \geq m_\varphi \vol_n([x^\delta,K]_R\setminus K).
\end{equation*}
We put $\tilde{\delta}=\frac{\delta}{m_\varphi}$ and note that the previous inequality implies that $x^\delta \in I_{R}^{\tilde{\delta}} (K)$.  We set
\begin{equation*}
\tilde{x}^{\tilde{\delta}} =\{\alpha x:\alpha\geq 0 \}\cap \partial (I_{R} ^{\tilde{\delta}}(K)).
 \end{equation*}
Then $\|x^\delta\| \leq \|\tilde{x}^{\tilde{\delta}}\|$  and thus by Lemma \ref{bounded} for $\delta < \min\{ \delta_0, \delta_1\}$, 
\begin{eqnarray}\label{bounded2}
\frac{1}{n} \frac{\langle x, N_K(x) \rangle}{\delta^\frac{2}{n+1}} \left[\left(\frac{\|x^\delta\|}{\|x\|} \right)^n-1 \right]  \leq 
\frac{1}{n} \frac{\langle x, N_K(x) \rangle}{m_\varphi^\frac{2}{n+1} \tilde{\delta}^\frac{2}{n+1}} \left[\left(\frac{\|\tilde{x}^{\tilde{\delta}}\|}{\|x\|} \right)^n-1 \right] \leq  \gamma_n \, r_{K}(x)^{-\frac{n-1}{n+1}},
\end{eqnarray}
where $\gamma_n$ and $\delta_0$ are as in Lemma \ref{bounded}.
With (\ref{bounded2}) and  Lebesgue Dominated Convergence Theorem, we get as the next step of the proof of Theorem \ref{theorem:weighted limit}
\begin{equation*}\label{proofstep2}
\lim_{\delta \to 0} \frac{\vol_n\left(I_{R, \varphi} ^\delta(K)\right)-\vol_n(K) } {\delta^\frac{2}{n+1}} = \frac{1}{n}  \int_{\partial K} \lim_{\delta \to 0} \frac{\langle x, N_K(x) \rangle}{\delta^\frac{2}{n+1}} \left[\left(\frac{\|x^\delta\|}{\|x\|}\right)^n -1 \right] \, d \mu_K(x). 
\end{equation*}
\vskip 2mm
\noindent
Now we  show the analog of Lemma \ref{limit}.  For  $x \in \partial K$, let $x^\delta$ be as in (\ref{xdelta2}). As in (\ref{proofstep3}), we write 
\begin{eqnarray*}
\frac{\langle \frac{x}{\|x\|}, N_K(x) \rangle}{\delta^\frac{2}{n+1}} \|x^\delta-x\| &\leq&\frac{1}{n} \frac{\langle x, N_K(x) \rangle}{\delta^\frac{2}{n+1}} \left[\left(\frac{\|x^\delta\|}{\|x\|}\right)^n -1 \right] \nonumber \\
& \leq &
\frac{\langle \frac{x}{\|x\|}, N_K(x) \rangle}{\delta^\frac{2}{n+1}} \|x^\delta-x\| \left(1+ \beta_n \frac{\|x^\delta-x\|}{\|x\|}\right),
\end{eqnarray*}
where $\beta_n$ is a constant. 
Let $\varepsilon >0$ be given. As $\varphi: U \to (0, \infty)$ is continuous, there exists an open neighborhood $V$ of $x$ such that for all $y \in V$,
\begin{equation}\label{stetig}
(1-\varepsilon) \, \varphi(x) \leq \varphi (y) \leq (1+\varepsilon) \, \varphi(x). 
\end{equation}
We choose $\delta$ so small that  $[x^\delta,K]_R\setminus K \subset V$. Then, by (\ref{nachunten}), 
\begin{eqnarray*}
 \delta &= & \int_{[x^\delta,K]_R\setminus K} \varphi(y) \, dy \geq (1-\varepsilon) \, \varphi(x)\, \vol_n([x^\delta,K]_R\setminus K)\\
 & \geq &
\frac{(1 - c\,  \varepsilon) \, 2^\frac{n+1}{2} \vol_{n-1}(B^{n-1}_2)}{n+1} \, \varphi(x)\,
 \left[ \langle \frac{x}{\|x\|}, N_K (x) \rangle \|x-x^\delta\| \right]^\frac{2}{n+1}
\left(\prod_{i=1}^{n-1} \left(\frac{a_n}{a_i^2} -\frac{1}{R}\right)\right)^{-\frac{1}{2}}, 
\end{eqnarray*}
or, equivalently, 
\begin{eqnarray*}
\frac{\langle \frac{x}{\|x\|}, N_K (x) \rangle \|x-x^\delta\|}{\delta^\frac{2}{n+1}} &\leq&(1 + c\, \varepsilon) \frac{1}{2} \left [ \frac{n+1}{\vol_{n-1}(B^{n-1}_2}\right]^\frac{2}{n+1} \, 
\varphi(x)^{-\frac{2}{n+1}}\, \left(\prod_{i=1}^{n-1} \left(\frac{a_n}{a_i^2} -\frac{1}{R}\right)\right)^{\frac{1}{n+1}}\\
&=&(1 + c\, \varepsilon) \frac{1}{2} \left [ \frac{n+1}{\vol_{n-1}(B^{n-1}_2}\right]^\frac{2}{n+1}\, \varphi(x)^{-\frac{2}{n+1}}\,  \left(\prod_{i=1}^{n-1} \left(\kappa_i(x,K) -\frac{1}{R}\right)\right)^{\frac{1}{n+1}}.
\end{eqnarray*}
The estimate from below is done similarly, using again (\ref{stetig}) and Lemma \ref{limit}.
\vskip 2mm
\noindent
This finishes the proof of Theorem \ref{theorem:weighted limit}.

\end{document}